\documentclass[11pt, a4paper]{amsart}

\usepackage{amsfonts}
\usepackage{comment}
\usepackage{amssymb}
\usepackage{mathtools}
\usepackage{amsmath, amsthm,color}
\usepackage{hyperref}
\usepackage{pdfsync}
\usepackage{bbm, dsfont}
\usepackage[margin=0.96in]{geometry}

\newtheorem{theorem}{Theorem}[section]
\newtheorem{corollary}[theorem]{Corollary}
\newtheorem{lemma}[theorem]{Lemma}

\numberwithin{equation}{section}

\RequirePackage[normalem]{ulem} 
\RequirePackage{color}\definecolor{RED}{rgb}{1,0,0}\definecolor{BLUE}{rgb}{0,0,1} 

\begin{document}

\title{Exponential integrability for log-concave measures}
\author[P. Ivanisvili, R. Russell]{Paata Ivanisvili and Ryan Russell}
\thanks{}
\address{Department of Mathematics,  North Carolina State University, UC Irvine}
\email{pivanisv@uci.edu \textrm{(P.\ Ivanisvili)}}

\address{Department of Mathematics and Statistics, California State University, Long Beach}
\email{rrussell14@g.ucla.edu \textrm{(R.\ Russell)}}
\makeatletter
\@namedef{subjclassname@2010}{
  \textup{2010} Mathematics Subject Classification}
\makeatother
\begin{abstract} 
Talagrand showed that finiteness of $\mathbb{E}\,  e^{\frac{1}{2}|\nabla f(X)|^{2}}$ implies finiteness of $\mathbb{E}\,  e^{f(X)}$ where $X$ is the standard Gaussian vector in $\mathbb{R}^{n}$ and $f$ is a smooth function with zero average. However, in this paper we show that finiteness of $ \mathbb{E}\,  e^{\frac{1}{2}|\nabla f|^{2}} (1+|\nabla f|)^{-1}$ implies finiteness of $\mathbb{E}\,  e^{f(X)}$, and we also obtain quantitative bounds  
 \begin{align*}
\log\, \mathbb{E}\,  e^{f} \leq 10\,  \mathbb{E}\,  e^{\frac{1}{2}|\nabla f|^{2}} (1+|\nabla f|)^{-1}.
 \end{align*} 
Moreover, the extra factor $(1+|\nabla f|)^{-1}$ is the best possible  in the sense that there is  smooth $f$ with  $\mathbb{E}\,  e^{f} =\infty$
  but  $\mathbb{E}\,  e^{\frac{1}{2}|\nabla f|^{2}} (1+|\nabla f|)^{-c}<\infty$ for all $c>1$. As an application we show corresponding dual inequalities for the discrete time dyadic martingales and its quadratic variations.
\end{abstract}
\maketitle

\section{Introduction}

In \cite{BG}, Bobkov--G\"otze showed that  for a smooth  $f:\mathbb{R}^{n} \to \mathbb{R}$  with $\mathbb{E} f(X)=0$ we have

\begin{align}\label{exp}
\mathbb{E}\,  e^{f(X)} \leq \left(\mathbb{E}\,  e^{\alpha |\nabla f(X)|^{2}}  \right)^{\frac{1}{2\alpha-1}} \quad \text{for any} \quad \alpha >1/2, 
\end{align}
for a class of random vectors $X$ in $\mathbb{R}^{n}$ satisfying log-Sobolev inequality {\em with constant 1}. In particular, the estimate (\ref{exp}) holds true when   $X\sim \mathcal{N}(0, \mathrm{I_{n\times n}})$ is the standard  Gaussian vector in $\mathbb{R}^{n}$  and $\mathrm{I_{n \times n}}$ is the identity matrix. The inequality  implies the measure concentration phenomena $\mathbb{P}(f(X)> \lambda) \leq e^{-\lambda^{2}/2}$ for all $\lambda \geq 0$ provided that $|\nabla f|\leq 1$ and $\mathbb{E}\,  f(X)=0$.  In~\cite{BG} it was  asked what happens in the endpoint case when $\alpha=1/2$, i.e.,  does finiteness of $\mathbb{E} \, e^{\frac{|\nabla f(X)|^{2}}{2}}$  imply finiteness of $\mathbb{E}\,  e^{f(X)}$ even for $n=1$ and $X \sim \mathcal{N}(0,1)$? 

 Notice that the  finiteness of $\mathbb{E} \, e^{\beta |\nabla f(X)|^{2}}$ for some $\beta \in (0, 1/2)$ does not imply finiteness of $\mathbb{E} \, e^{f(X)}$ (consider $X \sim \mathcal{N}(0,1)$ and $f(x) = \frac{x^{2}-1}{2}$). Therefore, perhaps  
 $$
 \mathbb{E}\,  e^{f(X)} < h\left(\mathbb{E}\,  e^{\frac{1}{2}|\nabla f(X)|^{2}}\right)
 $$
 is the ``best possible'' inequality one may ask  for some $h : [1, \infty) \to [0, \infty)$. 

It is not hard to see (see~\cite{BG}) that  Bobkov--G\"otze's exponential inequality (\ref{exp}) is optimal in terms of the powers, i.e., one cannot replace $\frac{1}{2\alpha-1}$ with $\frac{1}{c \alpha-1}$ for some $c<2$, and one cannot replace $e^{\alpha |\nabla f|^{2}}$ with $e^{c \alpha |\nabla f|^{2}}$ for some $c<1$. 
 
 According to a  discussion on page 8  in \cite{BG} Talagrand showed that even though (\ref{exp})  ``fails'' at the endpoint exponent $\alpha=\frac{1}{2}$, surprisingly,  the finiteness of  $\mathbb{E} \, e^{\frac{1}{2} |\nabla f(X)|^{2}}$ still implies finiteness of  $\mathbb{E}\,  e^{f}$ for $X \sim \mathcal{N}(0,\mathrm{I_{n \times n}})$.  We are not aware of Talagrand's proof as it was never published, we do not know if he solved the problem only for $n=1$ or for all $n\geq 1$. 
 
 In this paper we show  that the finiteness of $\mathbb{E}\,  e^{\frac{1}{2}|\nabla f|}(1+|\nabla f(X)|)^{-1}$ implies the finiteness of $\mathbb{E}\,  e^{f(X)}$ for all $n \geq 1$, and the extra factor $(1+|\nabla f|)^{-1}$ is the best possible in the sense that it cannot be replaced by $(1+|\nabla f|)^{-c}$ for some $c>1$. Moreover, we provide quantitive bounds.
 \begin{theorem}\label{mthii}
 For any $n\geq 1$ we have  
 \begin{align}\label{oo1}
 \log\, \mathbb{E}\,  e^{f(X)-\mathbb{E}f} \leq 10\,  \mathbb{E}\,  e^{\frac{1}{2}|\nabla f(X)|^{2}} (1+|\nabla f(X)|)^{-1}  
 \end{align} 
for all $f \in C^{\infty}_{0}(\mathbb{R}^{n})$ where $X \sim \mathcal{N}(0, \mathrm{I_{n \times n}})$. 
 \end{theorem}

 To see the sharpness of the factor $(1+|\nabla f|)^{-1}$ in (\ref{oo1}) let $n=1$, and let $f(x) = \frac{x^{2}}{2}$.   Then $\mathbb{E} \, e^{f(X)}  = \infty$. On the other hand, $\mathbb{E}\,  e^{\frac{1}{2}|\nabla f(X)|^{2}} (1+|\nabla f(X)|)^{-c}  = \frac{1}{\sqrt{2\pi}} \int_{\mathbb{R}} \frac{dx}{(1+|x|)^{c}} <\infty$ for all $c>1$.  It remains to multiply $f$ by a smooth cutoff function $\mathbbm{1}_{|x|\leq R}$ and take the limit $R \to \infty$.

 Using standard mass transportation arguments the exponential integrability (\ref{oo1}) may be extended to random vectors $X$ having  log-concave densities. 
 \begin{corollary}\label{sledoval}
 Let $X$ be an arbitrary random vector in $\mathbb{R}^{n}$ with density $e^{-u(x)}dx$ such that  $\mathrm{Hess}\, u \geq R\,  \mathrm{I_{n \times n}}$ for some $R>0$. Then 
 \begin{align}\label{coriu}
  \log\, \mathbb{E}\,  e^{f(X)-\mathbb{E}f} \leq 10\,  \mathbb{E}\,  e^{\frac{1}{2R}|\nabla f(X)|^{2}} (1+R^{-1/2}|\nabla f(X)|)^{-1}
 \end{align}
 holds for all $f \in C_{0}^{\infty}(\mathbb{R}^{n})$.
 \end{corollary}

Exponential integrability has been studied for other random vectors $X$ as well. Let us briefly record some known results where we  assume $f$ to be real valued with $\mathbb{E} f(Y)=0$. In all examples $Y$ is uniformly distributed on the set where it is given. 

\begin{align}
& \log \, \mathbb{E}\,  e^{f(Y)} \leq \mathbb{E}\, \frac{1}{4}|\nabla_{\mathbb{S}^{2}} f(Y)|^{2},   \qquad Y \in \mathbb{S}^{2}=\{\|x\|=1, x \in \mathbb{R}^{3}\},  \label{Trud}\\
& \log \, \mathbb{E}\,  e^{f(Y)} \leq 1+ \mathbb{E}\frac{1}{16}|\nabla f(Y)|^{2},   \quad  Y \in \mathbb{D}=\{\|x\|\leq 1, x \in \mathbb{R}^{2}\}, \label{Trud2}\\
& \log \, \mathbb{E} \, e^{f(Y)} \leq \log \,  \mathbb{E}\, e^{D(f)^{2}(Y)},   \qquad  \quad Y \in  \{-1,1\}^{n}, \label{Bobk}\\
& \log \, \mathbb{E} \, e^{f(Y)} \leq \log \,  \mathbb{E}\, e^{4|\nabla f|^{2}(Y)},   \qquad  \quad Y \in  [-1,1]^{n}, \quad \text{(only for convex $f$),} \label{Bobk2}
\end{align}
where in (\ref{Bobk}) by the symbol $D(f)^{2}$ we denote ``discrete gradient'' (see \cite{BG}). The estimate (\ref{Trud}), also known as Mozer--Trudinger inequality (with the best constants due to Onofri), has been critical for geometric applications \cite{MOS11, MOS12}. The exponential integrability  (\ref{Trud2}) is due to Carleson--Chang~\cite{CAR}. The estimates (\ref{Bobk}) and (\ref{Bobk2}) are due to Bobkov--G\"otze \cite{BG}. A slightly weaker version of (\ref{Bobk}), namely, $\mathbb{E}\, e^{f(Y)} \leq  \mathbb{E}\, e^{\frac{\pi^{2}}{8} D(f)^{2}(Y)}$ was obtained by Ben-Efraim--Lust-Piquard~\cite{Lust}.

The proof of the main theorem uses {\em heat--flow} arguments. We construct a certain increasing quantity $A(s)$  with respect to a parameter $s \in [0, 1]$. We will see that 
\begin{align*}
\log \mathbb{E}\, e^{f(X)} = A(0) \leq A(1)  \leq \mathbb{E}\, f(X)+ 10\, \mathbb{E}\,  e^{\frac{1}{2}|\nabla f(X)|^{2}} (1+|\nabla f(X)|)^{-1}. 
\end{align*}
To describe the expression for $A(t)$,  let $\Phi(t)=\mathbb{P}(X_{1}\leq t)$ be the Gaussian cumulative distribution function, and set $k(x) = - \log \frac{\Phi'(t)}{\Phi(t)} $.  Our main object will be a certain function $F: [0, \infty) \to [0, \infty)$ defined as
\begin{align}\label{ode2}
F(x) = \int_{0}^{x}e^{k((k')^{-1}(t))}dt \quad  \text{for all} \quad  x \in [0,\infty), 
\end{align}
where $(k')^{-1}$ is the inverse function to $k'$ (it will be explained in the next section why $F$ is well defined). 
For $g :\mathbb{R}^{n} \to (0, \infty)$ we consider its  heat flow  $U_{s} g(y) := \mathbb{E} g(y  + \sqrt{s} X)$ where $s \in [0,1]$. Then 
\begin{align*}
A(s) := U_{s}\left[ \log\,  U_{1-s} g  + F\left( \frac{\sqrt{s}\,  |\nabla U_{1-s} g|}{U_{1-s} g}\right) \right](0)
\end{align*}
will have the desired properties: $A'(s) \geq  0$, $A(0) = \log\,  \mathbb{E}\,  g$, and $A(1) = \mathbb{E}\, \log g+ \mathbb{E} F(\frac{|\nabla g|}{g})$. The argument gives the inequality 

\begin{align}\label{fin999}
\log \mathbb{E} g - \mathbb{E} \log g \leq \mathbb{E}\, F\Big(\frac{|\nabla g|}{g}\Big). 
\end{align}
If we set $g(x)= e^{f(x)}$ with $f :\mathbb{R}^{n} \to \mathbb{R}$ and use the chain rule $\frac{|\nabla g|}{g} = |\nabla f|$ we obtain
\begin{align}\label{fin99}
\log\,  \mathbb{E}\, e^{f-\mathbb{E}f}\leq \mathbb{E}\, F(|\nabla f|). 
\end{align}
The last step is to show the pointwise estimate $F(s) \leq 10\, e^{\frac{s^{2}}{2}}(1+s)^{-1}$ for all $s\geq 0$. We remark that the obtained inequality (\ref{fin99}) is stronger than (\ref{oo1}) and it should be considered as a corollary of (\ref{fin99}), however, due to a complicated expression for $F$ we decided to state the main result in the form (\ref{oo1}).

The computation of $A'(s)$ is technical and it is done in Section~\ref{bolo88}, where we also explain how the expression $A(t)$ was ``discovered''. We should note that the main reason that makes $A' \geq 0$ is the fact that $k'/k''>0$, and the inequality 
\begin{align*}
1-k''-k'e^{k} \geq 0,
\end{align*}
which for  $k= - \log \frac{\Phi'(t)}{\Phi(t)}$ serendipitously turns out to be equality. 

Sections~\ref{oney} and \ref{Mong} are technical and can be skipped when reading the paper for the first time. In these sections we show that $F \in C^{2}([0, \infty))$ is an increasing convex function with values  $F(0)=F'(0)=0$, $F''(0)=1$. Furthermore,  the ``modified'' hessian matrix of
\begin{align}\label{mfun}
M(x,y) := \log x + \int_{0}^{y/x} e^{k((k')^{-1}(t))}dt
\end{align}
  is positive semidefinite 
\begin{align}\label{kakun}
\begin{pmatrix}
M_{xx}+\frac{M_{y}}{y} & M_{xy} \\
M_{xy} &  M_{yy}
\end{pmatrix} \geq 0 \quad \text{for all} \quad (x,y) \in (0, \infty)\times [0, \infty).
\end{align}
In Section~\ref{bolo88} we demonstrate that the condition (\ref{kakun}) implies the  inequality 
\begin{align}\label{eshe}
 M(\mathbb{E}g(X), 0)\leq \mathbb{E} M(g(X), |\nabla g(X)|)
\end{align}
 for all smooth bounded $g :\mathbb{R}^{n} \to (0, \infty)$.  At the end of Section~\ref{bolo88}, we deduce Theorem~\ref{mthii} and Corollary~\ref{sledoval} from (\ref{eshe}).

As an application  in Section~\ref{prim} we show that the dual inequality to (\ref{fin999}), in the sense of duality described in Section 3.2 of \cite{INV2}, 
corresponds to 
\begin{theorem}\label{mt3}
For any positive martingale $\{\xi_{n}\}_{n \geq 0}$ on a probability space $([0,1], \mathcal{B}, dx)$ adapted to a discrete time dyadic  filtration $([0,1), \emptyset) =\mathcal{F}_{0}\subset \mathcal{F}_{1}\subset \ldots$ such that $\xi_{N}=\xi_{N+1}=\ldots=\xi_{\infty}>0$ for a sufficiently large $N$, we have  
\begin{align}\label{martin}
\log \mathbb{E} \xi_{\infty} - \mathbb{E} \log \xi_{\infty}  \leq \mathbb{E}\, G\Big(\frac{\xi_{\infty}}{[\xi_{\infty}]^{1/2}}\Big), 
\end{align}
where  $[\xi_{\infty}] = \sum_{k \geq 0} (\xi_{k+1}-\xi_{k})^{2}$ is the quadratic variation, and $G(t) := \int_{t}^{\infty}\int_{s}^{\infty}r^{-2}e^{\frac{s^{2}-r^{2}}{2}}drds$.
\end{theorem}
In Lemma~\ref{mokla1} we obtain the two sided estimate
\begin{align*}
\frac{1}{3} \log(1+t^{-2})\leq G(t) \leq \log(1+t^{-2}) \quad \text{for all} \quad t\geq 0.
\end{align*}
In particular, (\ref{martin}) implies 
\begin{align}\label{sled2}
\log \mathbb{E} \xi_{\infty} - \mathbb{E} \log \xi_{\infty} \leq \mathbb{E} \log\left( 1+\frac{[\xi_{\infty}]}{\xi_{\infty}^{2}}\right).
\end{align}
The estimate (\ref{sled2}) shows how well $\log \xi_{\infty}$ is concentrated around $\log \mathbb{E} \xi_{\infty}$  provided that  one can control the quadratic variation of $\xi_{\infty}$.  Theorem~\ref{mt3}  posits a duality approach developed in \cite{INV2}. This may be considered as complementary to the $e$-entropy bound  $\mathbb{E}e^{\xi_{\infty}- \mathbb{E} \xi_{\infty}} \leq \frac{e^{-\varepsilon}}{1-\varepsilon}$ which holds for all discrete time simple martingales $\xi_{n}$ (not necessarily positive and dyadic) provided that 
$[\xi_{\infty}] \leq \varepsilon^{2}$, see Corollary 1.12 in \cite{stolyar}. 

The proof of (\ref{martin})  uses the special  function 
$$
N(p,t) := \log(p) + \int_{p/\sqrt{t}}^{\infty}\int_{s}^{\infty}r^{-2}e^{\frac{s^{2}-r^{2}}{2}}drds
$$
  which we find by dualizing  $M(x,y)=\log x +F(y/x)$. We unfurl that $N$ is {\em heat convex}, i.e., 
\begin{align}\label{hcon}
2N(p,t) \leq N(p+a, t+a^{2})+N(p-a, t+a^{2})
\end{align}
holds for all reals $p,a,t$ such that $p\pm a \geq 0$ and $t\geq 0$. Finally, we reveal that after iterating (\ref{hcon}) we reap (\ref{martin}).

\section{Proof of Theorem~\ref{mthii} and Copollary~\ref{sledoval}}
\subsection{Step 1. An implicit function $F$ and its properties}\label{oney}
Let
\begin{align*}
k(x) := - \log\,  (\log\, \Phi(x))' =  \frac{x^{2}}{2}+\log\, \left( \int_{-\infty}^{x} e^{-\frac{s^{2}}{2}}ds\right)\quad \text{for all} \quad x \in \mathbb{R}. 
\end{align*} 
Define a real valued function $F$ implicitly as 
\begin{align}\label{fordf}
F(k'(t))  = \int_{-\infty}^{t}k''(s) e^{k(s)} ds \quad \text{for all} \quad t \in \mathbb{R}.
\end{align}

\begin{lemma}\label{prop01}
We have 
\begin{itemize}
\item[1.] $k'(-\infty)=0$; $k'(x) \sim x$ as $x \to \infty$; $k''>0$\,  \textup{(}and hence $k'>0$\textup{)}; 
\item[2.] $F :[0, \infty) \to [0,\infty)$; $F(0)=F'(0)=0$, $F''(0)=1$; $F'(k')=e^{k}$;  $F''(k') = \frac{k'}{k''} e^{k}$. 
\end{itemize}

\end{lemma}
\begin{proof}
Let us investigate the asymptotic behavior of $k$ and its derivatives at $x = -\infty$. Let $x<0$. For $m \geq 0,$ define $I_{m} :=e^{x^{2}/2}\int_{-\infty}^{x}e^{-s^{2}/2} s^{-m}ds$. Integration by parts reveals $I_{m} = -x^{-(m+1)}-(m+1)I_{m+2}$.  By iterating we obtain 
\begin{align*}
e^{\frac{x^{2}}{2}}\int_{-\infty}^{x}e^{-\frac{s^{2}}{2}}ds &= I_{0}=-x^{-1}+x^{-3}-3\cdot x^{-5}+3\cdot 5 \cdot  x^{-7}+3\cdot 5 \cdot 7 \cdot I_{8}\\
& = -x^{-1}+x^{-3}-3x^{-5}+O(x^{-7}) \quad \text{as} \quad x \to \infty
\end{align*}
because $|I_{8}| \leq \int_{-\infty}^{x} s^{-8}ds = O(x^{-7}).$  Thus, as $x \to -\infty$ we have 
\begin{align*}
&e^{k(x)} = I_{0} = -x^{-1}+x^{-3}-3x^{-5}+O(x^{-7}), \quad  e^{-k(x)} = -x-x^{-1}+2x^{-3}+O(x^{-5}),\\[8pt]
&k'(x) = x+e^{-k(x)}=-x^{-1}+2x^{-3}+O(x^{-5}), \quad k''(x) = 1-k'(x)e^{-k(x)} = x^{-2}+ O(x^{-4}),\\[8pt]
&k''(x)e^{k(x)}=-x^{-3} + O\left(x^{-5}\right), \quad \text{and} \quad \frac{k'(x)e^{k(x)}}{k''(x)}=1+O(x^{-2}).
 \end{align*}
The claim  $k'(x) = x + \frac{1}{e^{x^{2}/2}\int_{-\infty}^{x} e^{-s^{2}/2}ds} \sim x$ as $x \to \infty$ is trivial. Next, we show that $k''>0$.  
\begin{align*}
k'' =1- \frac{xe^{x^{2}/2}\int_{-\infty}^{x} e^{-s^{2}/2}ds +1}{(e^{x^{2}/2}\int_{-\infty}^{x} e^{-s^{2}/2}ds)^{2}}  = \frac{e^{x^{2}}}{e^{2k(x)}}\left[\left(\int_{-\infty}^{x} e^{-s^{2}/2}ds\right)^{2} - xe^{-x^{2}/2}\int_{-\infty}^{x}e^{-s^{2}/2}ds - e^{-x^{2}} \right].
\end{align*}
If we let $h(x) :=e^{-x^{2}/2}$, and $H(x) := \int_{-\infty}^{x} e^{-t^{2}/2}dt$, then it suffices to show 
\begin{align*}
u(x):=H^{2}-xhH - h^{2}>0.
\end{align*} 
Clearly $H'=h$, and $h' = -x h$. Next 
\begin{align*}
u' &= 2H h - h H + x^{2} h H - x h^{2} +2x h^{2} = H h + x^{2} h H +x h^{2}\\
&=(H + x^{2} H +x h)h = \left(H +h\frac{x}{1+x^{2}}\right) (1+x^{2})h.
\end{align*}
Let $v(x) = H +h\frac{x}{1+x^{2}}$. Then, we have  
\begin{align*}
v'(x) =h -h \frac{x^{2}}{1+x^{2}} + h  \frac{1-x^{2}}{(1+x^{2})^{2}} = \left( \frac{1}{1+x^{2}}  + \frac{1-x^{2}}{(1+x^{2})^{2}}\right)h = \frac{2}{(1+x^{2})^{2}}h >0.
\end{align*}
Since $v(-\infty)=0$ and $v'>0$, we obtain $v(x)>0$ for all $x \in \mathbb{R}$. In particular $u'>0$; taking into account that  $u(-\infty)=0$, we conclude $u(x)>0$ for all $x \in \mathbb{R}$.

To verify the second part of the lemma notice that by consider $t \to -\infty$ in (\ref{fordf}) we see $F(0)=0$. Taking the derivative in $t$ of (\ref{fordf}) and dividing both sides by $k''>0$ we obtain $F'(k) = e^{k}$. Considering $t \to -\infty$ we realize $F'(0)=0$. Taking the derivative of (\ref{fordf}) the second time we get $F''(k') = \frac{k'}{k''}e^{k}$. Since $\frac{k'}e^{k}{k''} = 1+O(x^{-2})$ as $x \to -\infty$ we obtain $F''(0)=1$. The lemma is proved.

\end{proof}

It follows that  $k'>0$ and $k' : \mathbb{R} \to [0, \infty)$. Thus, we may consider the inverse map $t \mapsto k'(t)$ denoted as $(k')^{-1} :[0, \infty) \to \mathbb{R}$. After making a change of variables in (\ref{fordf}), we write 
\begin{align*}
F(x) = \int_{-\infty}^{(k')^{-1}(x)}k''(s) e^{k(s)}ds \stackrel{s=(k')^{-1}(u)}{=} \int_{0}^{x}e^{k((k')^{-1}(u))}du, 
\end{align*}
which coincides with the expression announced in (\ref{ode2}).

\begin{lemma}\label{poinaa}
We have $F(x) \leq 10\,  e^{\frac{x^{2}}{2}}(1+x)^{-1}$ for all $x \geq 0$. 
\end{lemma}

\begin{proof}
Notice that $k'(u) = u+e^{-k(u)} \geq u$ (for all $u \in \mathbb{R}$) and $k''>0$. Therefore, $u \geq (k')^{-1}(u)$ for $u \geq 0$; so $k(u) \geq k((k')^{-1}(u))$ due to the fact that $k'>0$. Thus 
\begin{align*}
F(x) \leq \int_{0}^{x}e^{k(u)}du = \int_{0}^{x}e^{\frac{u^{2}}{2}}\int_{-\infty}^{u}e^{-\frac{s^{2}}{2}}ds du \leq \sqrt{2\pi} \int_{0}^{x} e^{\frac{u^{2}}{2}}du.
\end{align*}
Next, we claim the simple  chain of inequalities 
\begin{align*}
\int_{0}^{x} e^{\frac{u^{2}}{2}}du \stackrel{(A)}{\leq} \frac{2x}{1+x^{2}} e^{\frac{x^{2}}{2}} \stackrel{(B)}{\leq} \frac{3}{1+x} e^{\frac{x^{2}}{2}}.
\end{align*}
Indeed, inequality (A) follows from the fact that it is true at $x=0$, and 
\begin{align*}
\frac{d}{dx} \left( \frac{2x}{1+x^{2}} e^{\frac{x^{2}}{2}} - \int_{0}^{x}e^{\frac{u^{2}}{2}}du \right) = e^{\frac{x^{2}}{2}}\left( 1-\frac{4x^{2}}{(1+x^{2})^{2}}\right)\geq e^{\frac{x^{2}}{2}}\left( 1-\frac{4x^{2}}{(2x)^{2}}\right) =0.
\end{align*}
The inequality (B) is elementary.  Therefore, we conclude that 
\begin{align*}
F(x) \leq 3 \sqrt{2\pi} \,e^{\frac{x^{2}}{2}}(1+x)^{-1} \leq 10 \,e^{\frac{x^{2}}{2}}(1+x)^{-1} \quad \text{for all} \quad x\geq 0. 
\end{align*}
\end{proof}

\subsection{Step 2. Monge--Amp\`ere type PDE}\label{Mong} Define  
\begin{align}\label{chemif}
M(x,y) = \log x + F(y/x) \quad \text{for all} \quad (x,y) \in (0, \infty)\times [0, \infty).
\end{align}
Clearly $M \in C^{2}$ and $M_{y}(x,0)=0$, where $M_{x} = \frac{\partial M}{\partial x}$ and $M_{y} = \frac{\partial M}{\partial y}$. Next, let us consider the matrix 
\begin{align}
A(x,y):=\begin{pmatrix}
M_{xx}+\frac{M_{y}}{y} & M_{xy}\\
M_{xy} & M_{yy}
\end{pmatrix}.
\end{align}
 We claim 
\begin{lemma}
For each $(x,y) \in  (0, \infty)\times [0, \infty)$ the matrix $A(x,y)$ is positive semidefinite with $\det(A)=0$. 
\end{lemma}
\begin{proof}
Let us calculate the partial derivatives of $M$. Let $t :=yx^{-1}$. We have 
\begin{align*}
M_{x} &= x^{-1}-yx^{-2}F'(yx^{-1}) = x^{-1}(1-tF'(t)),\\[8pt]
M_{xx} &=-x^{-2}+2yx^{-3}F'(yx^{-1})+(yx^{-2})^{2}F''(yx^{-1}) = x^{-2}(-1+2tF'(t)+t^{2}F''(t)),\\[8pt]
M_{y} &= x^{-1}F'(t),  \quad M_{yx}=-x^{-2}(F'(t)+tF''(t)),\\[8pt]
M_{yy} & = x^{-2}F''(t)\stackrel{\mathrm{Lemma~\ref{prop01}}}{>}0.
\end{align*}
To see that $A(x,y)$ is positive semidefinite  it suffices (due to the inequality $M_{yy}>0$) to check $\mathrm{det}(A)=0$. We have 
\begin{align*}
\det (A) &= M_{xx}M_{yy}-M_{xy}^{2} + \frac{M_{y}M_{yy}}{y}\\
&=x^{-4}\left[ (-1+2tF'+t^{2}F'')F'' - (F'+tF'')^{2} + \frac{F'F''}{t}\right]\\
&=x^{-4}\left[-F'' - (F')^{2} + \frac{F'F''}{t} \right].
\end{align*}
Next, for $t=k'$ by Lemma~\ref{prop01} have $F'(k')=e^{k}$ and $F''(k') = \frac{k'e^{k}}{k''}$. Therefore 
\begin{align*}
-F'' - (F')^{2} + \frac{F'F''}{t} = -\frac{k'e^{k}}{k''} - e^{2k} + \frac{e^{2k}}{k''} = \frac{e^{2k}}{k''} \left( 1 - k'' - k' e^{-k}\right)=0,
\end{align*}
due to the fact that  $k'(x) = x+ e^{-k(x)}$ (and hence $k'' = 1-k'e^{-k}$). 
\end{proof}
\subsection{Step 3: the heat-flow argument}\label{bolo88}

First we would like to give an explanation on how the flow is constructed.  For simplicity consider $n=1$. If we succeed in proving the inequality
\begin{align}\label{red1}
M(\mathbb{E} g(\xi), 0) \leq \mathbb{E} M(g(\xi), |g'(\xi)|), \qquad g: \mathbb{R} \to (0, \infty),
\end{align}
where $\xi \sim \mathcal{N}(0,1)$, and $M(x,y) = \log \, x + F(y/x)$, then we glean $\log\,\mathbb{E} g +F(0)\leq  \mathbb{E} \log\, g + \mathbb{E}F(|g'|/g)$, which for $g=e^{f}$ coincides with (\ref{fin99}). So, the question is to prove (\ref{red1}).  We consider discrete approximation of $\xi$, namely, let 
\begin{align*}
\vec{\varepsilon} =(\varepsilon_{1}, \ldots, \varepsilon_{m}),
\end{align*}
 where $\varepsilon_{j}$ are i.i.d. symmetric Bernoulli $\pm 1$ random variables. By the  Central Limit Theorem 
\begin{align*}
\frac{\varepsilon_{1}+\ldots+\varepsilon_{m}}{\sqrt{m}} \stackrel{d}{\to} \xi \quad \text{as} \quad m \to \infty. 
\end{align*}
We hope to prove the  ``hypercube analog'' of (\ref{red1}), i.e., 
\begin{align}\label{discan}
M(\mathbb{E}\, \tilde{g}(\vec{\varepsilon}), 0) \leq \mathbb{E}M(\tilde{g}(\vec{\varepsilon}), |D\tilde{g}(\vec{\varepsilon})|), \qquad \tilde{g}(\vec{\varepsilon}) = g\left(\frac{\varepsilon_{1}+\ldots+\varepsilon_{m}}{\sqrt{m}}\right),
\end{align}
for all $m\geq 1$, where the ``discrete'' gradient $|D\tilde{g}(\vec{\varepsilon})| : = \sqrt{\sum_{j=1}^{m} |D_{j} g(\vec{\varepsilon})|^{2}}$ is defined in a subtle way as  
\begin{align*}
D_{j}\tilde{g}(\varepsilon_{1}, \ldots, \varepsilon_{m}) = \frac{\tilde{g}(\varepsilon_{1}, \ldots, \varepsilon_{j}, \ldots, \varepsilon_{m}) -\tilde{g}(\varepsilon_{1}, \ldots, -\varepsilon_{j}, \ldots, \varepsilon_{m}) }{2} \quad \text{for} \quad j=1, \ldots, m. 
\end{align*}
One sees that as $m \to \infty$ we have  
\begin{align*}
D_{j}\tilde{g}(\vec{\varepsilon}) = g'\left(\frac{\varepsilon_{1}+\ldots+\varepsilon_{m}}{\sqrt{m}}\right) \frac{\varepsilon_{j}}{\sqrt{m}} + O\left(\frac{1}{m}\right)
\end{align*}
and 
\begin{align*}
|D \tilde{g}(\vec{\varepsilon})| = \sqrt{\left[g'\left(\frac{\varepsilon_{1}+\ldots+\varepsilon_{m}}{\sqrt{m}}\right) \right]^{2} + O\left(\frac{1}{\sqrt{m}}\right)}
\end{align*}
at least for bounded smooth $g$ with uniformly bounded derivatives. Thus taking the limit $m \to \infty$ we observe that the right hand side of (\ref{discan}) converges to the right hand side of (\ref{red1}),  in particular,  (\ref{discan}) implies (\ref{red1}).

Next, we take this one step further and consider the inequality (\ref{discan}) for all $\tilde{g} : \{-1,1\}^{m} \to \mathbb{R}$ instead of the specific ones defined in (\ref{discan}); in doing so we are ever so slightly enlarging the class of test functions   including those that are not invariant with respect to permutations of $(\varepsilon_{1}, \ldots, \varepsilon_{n})$. To prove 
\begin{align}\label{ook}
M(\mathbb{E} h, 0) \leq \mathbb{E} M(h, |Dh|) \quad \text{for all} \quad h :\{-1,1\}^{m} \to (0, \infty)\quad \text{and all} \quad m\geq 1, 
\end{align}
one trivial argument would be to invoke product structure of $\{-1,1\}^{m}$. For example, if we manage to show an {\em intermediate}  ``4-point'' inequality  
\begin{align}\label{two-pp}
M(\mathbb{E}_{\varepsilon_{1}} h, |D \mathbb{E}_{\varepsilon_{1}}h|) \leq \mathbb{E}_{\varepsilon_{1}} M(h, |Dh|),
\end{align}
where $\mathbb{E}_{\varepsilon_{1}}$ averages only with respect to $\varepsilon_{1}$, then by iterating  (\ref{two-pp}) we salvage
\begin{align*}
M(\mathbb{E}h, 0) = M(\mathbb{E}_{\varepsilon_{m}} \ldots \mathbb{E}_{\varepsilon_{1}}h, |D\mathbb{E}_{\varepsilon_{m}} \ldots\mathbb{E}_{\varepsilon_{1}}h|)\leq \mathbb{E}_{\varepsilon_{1}}\ldots\mathbb{E}_{\varepsilon_{m}} M(h, |Dh|) = \mathbb{E} M(h, |Dh|).
\end{align*}
Upon closer inspection, we see that (\ref{two-pp}) follows\footnote{In fact they are equivalent provided that $y \mapsto M(x,y)$ is nondecreasing} from the following $4$ point inequality 
\begin{align}\label{two-pp1}
2M(x,y) \leq M(x+a, \sqrt{a^{2}+(y+b)^{2}}) + M(x-a, \sqrt{a^{2}+(y-b)^{2}})
\end{align}
for all real numbers $x,y,a,b$ such that $x \pm a>0$. To prove (\ref{two-pp1}) for one specific $M$ seems to be a possible task, however, if we take into account that $M$ is defined by (\ref{chemif}) which  involves implicitly defined  $F$, the four parameter inequality (\ref{two-pp1}) becomes complicated (see \cite{IV2} where one such inequality was proved for $M(x,y)=-\Re\,  (x+iy)^{3/2}$ by tedious computations involving high degree polynomials  with integer coefficients). 

Expanding (\ref{two-pp1}) at point $(a,b)=(0,0)$ via Taylor series  one easily obtains a necessary assumption, the infinitesimal form of (\ref{two-pp1}) i.e., 
\begin{align} \label{mat09}
\begin{pmatrix}
M_{xx}+\frac{M_{y}}{y} & M_{xy} \\
M_{xy} & M_{yy}
\end{pmatrix}\geq 0.
\end{align} 
Of course, the infinitesimal condition (\ref{mat09}) does not necessarily imply its global two-point inequality (\ref{two-pp1}) (and in particular (\ref{ook})). Also, it may seem implausible  to believe that the positive semidefiniteness (\ref{mat09}) implies the inequality (\ref{red1}) in gauss space.  Surprisingly this last guess turns out to be correct, and  perhaps the reason lies in the fact that one only needs to verify (\ref{discan}) as $m \to \infty$ (and only for symmetric functions $\tilde{g}$). Let us ``take the limit''  and see how the heat flow arises. 

Let $\mathbb{E}_{k}$ be  the average with respect to the first $\varepsilon_{1}, \ldots, \varepsilon_{k}$, and let $\mathbb{E}^{m-k}$ be the average with respect to the remaining variables $\varepsilon_{k+1}, \ldots, \varepsilon_{m}$. Then the 4-point inequality (\ref{two-pp}) implies 
\begin{align*}
k \mapsto \mathbb{E}^{k}M(\mathbb{E}_{m-k}\,\tilde{g},  |D \mathbb{E}_{m-k}\tilde{g}|) \quad \text{is nondecreasing on}\quad  0\leq k \leq m. 
\end{align*}

The expression $ \mathbb{E}^{k}M(\mathbb{E}_{m-k}\,\tilde{g},  |D \mathbb{E}_{m-k}\tilde{g}|)$ we rewrite as $\mathbb{E}^{k}M(A,B)$, where 
\begin{align*}
&A = \mathbb{E}_{m-k} g\left( \frac{\sum_{j=1}^{k} \varepsilon_{j}}{\sqrt{k}} \sqrt{\frac{k}{m}} + \frac{\sum_{j=k+1}^{m} \varepsilon_{j}}{\sqrt{m-k}} \sqrt{1-\frac{k}{m}}\right),\\
&B =  \sqrt{\frac{k}{m} \left[ \mathbb{E}_{m-k}g'\left(\frac{\sum_{j=1}^{k} \varepsilon_{j}}{\sqrt{k}} \sqrt{\frac{k}{m}}+\frac{\sum_{j=k+1}^{m} \varepsilon_{j}}{\sqrt{m-k}}  \sqrt{1-\frac{k}{m}}\right)\right]^{2}+ O\left(\frac{k}{m^{3/2}}\right)}.
\end{align*}

Taking $k,m \to \infty$ so that $\frac{k}{m} \to s \in [0,1]$ one can conclude 
\begin{align*}
s \mapsto \mathbb{E}_{X} M(\mathbb{E}_{Y} g(X\sqrt{s} + Y\sqrt{1-s}), \sqrt{s} | \mathbb{E}_{Y} g'(X\sqrt{s} +Y\sqrt{1-s})|) \quad \text{is nondecreasing on} \quad [0,1],
\end{align*}
where $X, Y \in \mathcal{N}(0,1)$ are independent, and $\mathbb{E}_{X}$ takes the expectation with respect to random variable $X$. In other words if we let $U_{s} g(y) = \mathbb{E} g(y+\sqrt{s} X)$ to be a heat flow defined as 

\begin{align*}
\frac{\partial}{\partial s} U_{s}g = \frac{1}{2}  \frac{\partial^{2}}{\partial x^{2}} \, U_{s}g, \quad U_{0}g=g. 
\end{align*}
then 
\begin{align}\label{map1}
s \mapsto U_{s}M(U_{1-s}g, \sqrt{s} |U_{1-s}g'|) \quad \text{is nondecreasing on} \quad  [0,1]. 
\end{align}

Luckily  we may ignore all the steps by starting start from the map (\ref{map1}) and  taking its derivative in $s$ to divine  when it has nonnegative sign. Slightly abusing the notations, denote $D = \frac{\partial}{\partial x}$, and, for simplicity, let us work with the map $s \mapsto U_{s}M(U_{1-s}g, \sqrt{s}\,  U_{1-s}g')$, where we omit the absolute value in the second argument of $M$. Let $b = U_{1-s}g$. Clearly $\frac{\mathrm{d}}{\mathrm{d}s} b = -\frac{1}{2}\,  D^{2} b$.
  We have 
\begin{align*}
\frac{\mathrm{d}}{\mathrm{d}s} U_{s}M(b, \sqrt{s} Db) &= \\
&= \frac{1}{2}\, D^{2}U_{s}M(b,\sqrt{s} Db) + U_{s} \left(-\frac{1}{2}\, D^{2}b\, M_{x}+\left(\frac{1}{2\sqrt{s}}Db - \frac{\sqrt{s}}{2} D^{3}b\right)M_{y} \right)\\
&=\frac{U_{s}}{2}\left(D(M_{x}Db+M_{y}\sqrt{s} D^{2}b) - M_{x}D^{2}b+\frac{M_{y}}{\sqrt{s}}Db-M_{y}\sqrt{s} D^{3}b \right) \\
&=\frac{U_{s}}{2}\left(M_{xx}(Db)^{2}+2M_{xy}\sqrt{s} Db\,  D^{2}b+M_{yy}s (D^{2}b)^{2} +\frac{M_{y}}{\sqrt{s}}Db\right).
\end{align*}
And notice that 
\begin{align*}
&M_{xx}D^{2}b+2M_{xy}\sqrt{s} Db\,  D^{2}b+M_{yy}s (D^{2}b)^{2} +\frac{M_{y}}{\sqrt{s}}Db=\\
&\begin{pmatrix} 
Db &\sqrt{s} D^{2}b
\end{pmatrix}
\begin{pmatrix}
M_{xx}+\frac{M_{y}}{\sqrt{s}Db} & M_{xy}\\
M_{xy} & M_{yy}
\end{pmatrix}
\begin{pmatrix}
Db\\
\sqrt{s}D^{2}b
\end{pmatrix}\geq 0.
\end{align*}
It remains to extend the argument to higher dimensions and put the the absolute value back into the second argument of $M$. 
\begin{theorem}
Let $M : (0, \infty)\times [0, \infty) \to \mathbb{R}$ be such that $M \in C^{2}$ with $M_{y}(x,0)=0$,  and 
\begin{align}\label{kuku}
\begin{pmatrix}
 M_{xx} + \frac{M_{y}}{y} & M_{xy}\\
 M_{xy} & M_{yy}
 \end{pmatrix} \geq 0.
\end{align}
Then the map 
\begin{align}\label{ai2}
s \mapsto U_{s}M(U_{1-s}g, \sqrt{s} |\nabla U_{1-s} g|) \quad \text{is nondecreasing on} \quad [0,1]
\end{align}
for all smooth bounded $g : \mathbb{R}^{n} \to (0, \infty)$ with uniformly bounded first and  second derivatives.  
\end{theorem}

\begin{proof}
Let $M(x,y) = B(x,y^{2})$. Let $B_{1}$ and $B_{2}$ be partial derivatives of $B$.  Positive semidefiniteness of the matrix (\ref{kuku}) in terms of $B$ converts to 
\begin{align}\label{vax01}
\begin{pmatrix}
B_{11}(x,y^{2})+2B_{2}(x,y^{2}) & 2yB_{12}(x,y^{2})\\
2yB_{12}(x,y^{2}) & 2B_{2}(x,y^{2})+4y^{2}B_{22}(x,y^{2})
\end{pmatrix}\geq 0
\end{align}
for all $x>0$, and all $y \geq 0$ (in fact for all $y \in \mathbb{R}$). Next, let $G = P_{1-s}g$. Clearly $\frac{\mathrm{d}}{\mathrm{ds}} G = -\frac{\Delta}{2} G$.
We have 
\begin{align*}
&\frac{\mathrm{d}}{\mathrm{ds}} U_{s}B(U_{1-s}g, s |U_{1-s}\nabla g|^{2}) = \\
&\frac{1}{2}\, U_{s} \left[ \Delta B(G, s|\nabla G|^{2})-B_{1}\Delta G +2B_{2} |\nabla G|^{2} - 2B_{2} s\nabla G \cdot \nabla \Delta G\right].
\end{align*}
Next, let $D_{j} = \frac{\partial }{\partial x_{j}}$. Then 
\begin{align*}
 D_{j}B(G, s|\nabla G|^{2}) &= B_{1} \, D_{j}G + B_{2}\, s D_{j} |\nabla G|^{2},\\[8pt]
 D_{j}^{2} B(G, s|\nabla G|^{2}) &=B_{11}(D_{j}G)^{2}+2\, B_{12} D_{j}G\, s D_{j}|\nabla G|^{2} +  B_{22}\, s^{2} (D_{j}|\nabla G|^{2})^{2}\\[8pt]
 &\quad + B_{1}D_{j}^{2}G + B_{2} \, s D_{j}^{2} |\nabla G|^{2},\\[8pt]
 \text{and} \quad \Delta B &= B_{11}|\nabla G|^{2} + 2B_{12} \nabla G \cdot s \nabla |\nabla G|^{2}+B_{22} \left|s\nabla |\nabla G|^{2}\right|^{2}\\[8pt]
 &\quad + B_{1} \Delta G + B_{2} s\Delta |\nabla G|^{2}. 
\end{align*}
Notice that $\Delta |\nabla G|^{2} = 2\nabla G\cdot \nabla \Delta G + 2\mathrm{Tr}\,  (\mathrm{Hess}\, G)^{2}$. Therefore
\begin{align}
&\Delta B(G, s|\nabla G|^{2})-B_{1}\Delta G +2B_{2} |\nabla G|^{2} - 2B_{2} s\nabla G \cdot \nabla \Delta G = \nonumber \\
&B_{11}|\nabla G|^{2} + 2B_{12} \nabla G \cdot s \nabla |\nabla G|^{2}+B_{22} \left|s\nabla |\nabla G|^{2}\right|^{2} + 2B_{2} |\nabla G|^{2} + 2B_{2}s \mathrm{Tr}\, (\mathrm{Hess}\, G)^{2} \geq\nonumber \\
& B_{11}|\nabla G|^{2}   - 2|B_{12}| |\nabla G|  \left|s \nabla |\nabla G|^{2}\right|+B_{22} \left|s\nabla |\nabla G|^{2}\right|^{2} + 2B_{2} |\nabla G|^{2} + 2B_{2}s \mathrm{Tr}\, (\mathrm{Hess}\, G)^{2}. \label{laki}
\end{align}
We observe the inequality 
\begin{align}\label{karoche}
\mathrm{Tr}\, (\mathrm{Hess}\, G)^{2} \, |\nabla G|^{2} = \sum_{j=1}^{n} |\nabla D_{j} G|^{2} |\nabla G|^{2} \geq \sum_{j=1}^{n}\left(\nabla D_{j}G \cdot \nabla G\right)^{2} = \frac{1}{4} \left|\nabla |\nabla G|^{2}\right|^{2}.
\end{align}
First we want to consider the case when $|\nabla G|=0$.  We recall $M(x,y)=B(x,y^{2})$. Therefore $B_{2}(x,0)$ exists and is equal to $\frac{1}{2} M_{yy}(x,0)$ (due to the fact that $M_{y}(x,0)=0$). Also 
$$
\lim_{y \to 0} B_{12}(x,y^{2})y=\frac{1}{2}M_{xy}(x,0),
$$
 and 
\begin{align*}
\lim_{y \to 0} B_{22}(x,y^{2})y^{2} =\lim_{y \to 0} \frac{1}{4}\left(M_{yy}(x,|y|)-2B_{2}(x,y^{2})\right) =0.
\end{align*}
Therefore, if $|\nabla G|=0$, then due to the inequality (\ref{karoche}), the expression (\ref{laki}) simplifies to 
\begin{align*}
2B_{2}(G, 0) s \mathrm{Tr}\, (\mathrm{Hess}\, G)^{2} = \frac{1}{2}M_{yy}(G, 0)  s \mathrm{Tr}\, (\mathrm{Hess}\, G)^{2} \geq 0,
\end{align*}
where the last inequality holds true by the assumption (\ref{kuku}), hence (\ref{laki}) is nonnegative. 

If $|\nabla G|> 0$ then we proceed as follows: the assumption (\ref{kuku}) implies $yM_{xx}+M_{y}\geq 0$. In particular taking $y=0$ we obtain $M_{y}(x,0)\geq 0$. Also it follows from (\ref{kuku}) that $M_{yy}\geq 0$. Thus $M_{y}(x,y)\geq 0$ for all $y\geq 0$. In particular $B_{2}(x,y^{2})\geq 0$ for all $y>0$ (and also for $y=0$ as we just noticed $B_{2}(x,0) = \frac{1}{2}M_{yy}(x,0)\geq 0$). Therefore, using (\ref{karoche}), we may estimate the last term in (\ref{laki}) from below as
$B_{2}  \frac{\left|s \nabla |\nabla G|^{2}\right|^{2}}{2s |\nabla G|^{2}}$. Finally, 
\begin{align*}
&B_{11}|\nabla G|^{2}   - 2|B_{12}| |\nabla G|  \left|s \nabla |\nabla G|^{2}\right|+B_{22} \left|s\nabla |\nabla G|^{2}\right|^{2} + 2B_{2} |\nabla G|^{2} + B_{2}  \frac{\left|s \nabla |\nabla G|^{2}\right|^{2}}{2s |\nabla G|^{2}}\\
&=
\begin{pmatrix} |\nabla G| & \frac{\sqrt{s}\left| \nabla |\nabla G|^{2}\right|}{2|\nabla G|} \end{pmatrix}
\begin{pmatrix}
B_{11}+2B_{2} & -2\sqrt{s} |\nabla G|\,  |B_{12}| \\
-2\sqrt{s} |\nabla G|\, |B_{12}| & 4s |\nabla G|^{2}B_{22} + 2B_{2}
\end{pmatrix}
\begin{pmatrix}
|\nabla G|\\
\frac{\sqrt{s}\left|\nabla |\nabla G|^{2}\right|}{2|\nabla G|}
\end{pmatrix}\geq 0
\end{align*}
by the assumption (\ref{vax01}) and the fact that $B$ is evaluated at point $(G, s|\nabla G|^{2})$.

\end{proof}

\vskip0.5cm

\noindent{\em Proof of Theorem~\ref{mthii}}.  Notice that $U_{s} g(y) = \mathbb{E}\,  g(y+\sqrt{s}X)$, therefore, comparing the values of the map (\ref{ai2}) at the endpoints $s=0$ and $s=1$ we obtain 
\begin{align*}
 M(\mathbb{E} g(X), 0)=U_{0}M(U_{1}g, \sqrt{0} |U_{1}\nabla g|)(0) \leq U_{1}M(U_{0}g,  |U_{0}\nabla g|)(0)= \mathbb{E} M(g(X), |\nabla g(X)|). 
\end{align*}
In particular,  for $g=e^{f}$ where $f \in C^{\infty}_{0}(\mathbb{R}^{n})$ we obtain 
\begin{align*}
\log\,  \mathbb{E}\,  e^{f(X)}  \leq \mathbb{E} f(X) + \mathbb{E} F(|\nabla f(X)|). 
\end{align*}
Finally, using the pointwise inequality $F(x) \leq 10\, e^{\frac{x^{2}}{2}}(1+x)^{-1}$ from Lemma~\ref{poinaa} finishes the proof of Theorem~\ref{mthii}  $\hfill\square$

\vskip0.9cm

\noindent{\em Proof of Corollary~\ref{sledoval}}.  Let $d\mu = e^{-u(x)}dx$ be the density of the log-concave random vector $X$ with $\mathrm{Hess}\, u \geq R\,  \mathrm{I_{n\times n}}$ for some $R>0$. It follows from \cite{CAF} that there exists a convex function $\psi :\mathbb{R}^{n} \to \mathbb{R}$ such that  the Brenier map $T = \nabla \psi$ pushes forward the gaussian measure $d\gamma_{n}(x)=\frac{e^{-|x|^{2}/2}}{\sqrt{(2\pi)^{n}}}dx$ onto $d\mu$ and $0\leq \mathrm{Hess}\, \psi \leq \frac{1}{\sqrt{R}}\mathrm{I_{n \times n}}$. Next, apply the inequality 
\begin{align}\label{samotxe}
\log\, \int_{\mathbb{R}^{n}} e^{f(x)} d\gamma_{n}(x) \leq \int_{\mathbb{R}^{n}}f(x) d\gamma_{n}(x) + \int_{\mathbb{R}^{n}} F(|\nabla f(x)|)d\gamma_{n}(x)
\end{align}
with $f(x) = h(\nabla \psi(x))$ for an arbitrary $h \in C^{\infty}_{0}(\mathbb{R}^{n})$. Then notice that 
\begin{align*}
|\nabla f(x)| = |\mathrm{Hess}\,\psi \nabla h(\nabla \psi(x))| \leq \frac{1}{\sqrt{R}} |\nabla h(\nabla \psi(x))|.
\end{align*}
Since $F'>0$ we conclude that 
$$
F(|\nabla f(x)|)\leq F\left( \frac{1}{\sqrt{R}} |\nabla h(\nabla \psi(x))| \right) \leq 10\, e^{\frac{|\nabla h(\nabla \psi(x))|^{2}}{2R}}(1+R^{-1/2}|\nabla h(\nabla \psi)|)^{-1}.
$$
The preceding inequality together with (\ref{samotxe}) implies 
\begin{align*}
\log\, \int_{\mathbb{R}^{n}} e^{h(x)} d\mu(x) \leq \int_{\mathbb{R}^{n}}h(x) d\mu(x)  +10\,  \int_{\mathbb{R}^{n}}  e^{\frac{|\nabla h(x)|^{2}}{2R}}(1+R^{-1/2}|\nabla h(x)|)^{-1}d\mu(x)
\end{align*}
for all $h \in C^{\infty}_{0}(\mathbb{R}^{n})$. This finishes the proof of Corollary~\ref{sledoval}.
  $\hfill\square$

\section{Applications: the proof of Theorem~\ref{mt3} and the estimate~(\ref{sled2})}\label{prim}

Let us recall the definition of dyadic martingales. 
For each $n \geq 0$ we denote by $\mathcal{D}_{n}$ dyadic intervals belonging to $[0,1)$ of level $n$, i.e., 
\begin{align*}
\mathcal{D}_{n}  =\left\{ \left[\frac{k}{2^{n}}, \frac{k+1}{2^{n}}\right), \quad k =0, \ldots, 2^{n}-1 \right\}.
\end{align*}
Given $\xi \in L^{1}([0,1],dx)$ define a {\em dyadic martingale} $\{ \xi_{k}\}_{k\geq 0}$ as  
 \begin{align*}
 \xi_{n}(x) := \sum_{I \in \mathcal{D}_{n}}  \langle \xi \rangle_{I} \mathbbm{1}_{I}(x), \quad n\geq 0, 
\end{align*}
where $\langle \xi \rangle_{I} = \frac{1}{|I|} \int_{I} \xi dx$, here $|I|$ denotes the Lebesgue length of $I$. If we let $\mathcal{F}_{n}$ to be the sigma algebra generated by the dyadic intervals in $\mathcal{D}_{n}$ then  $\xi_{n} = \mathbb{E} (\xi|\mathcal{F}_{n})$ is the martingale with respect to the increasing family of filtrations $\{ \mathcal{F}_{k}\}_{k \geq 0}$. Next we define the quadratic variation 
\begin{align*}
[\xi] =  \sum_{n\geq 0} d_{n}^{2},
\end{align*}
where $d_{n} := \xi_{n} - \xi_{n-1}$ is the martingale difference sequence. In what follows to avoid the issues with convergence of the infinite series we will be assuming that all but finitely many $d_{b}$ are zero,  i.e., $\xi_{N}=\xi_{N+1}=\ldots = \xi$  for $N$ sufficiently large. Such martingales we call simple dyadic martingales, they are also known as Walsh-Paley martingales \cite{hyt}. 

Let $N(p,t) := \log(p)+G(p/\sqrt{t})$ for $p>0$, $t \geq 0$ where
\begin{align*}
G(s) = \int_{s}^{\infty}\int_{r}^{\infty}u^{-2}e^{\frac{r^{2}-u^{2}}{2}}dudr, \quad s>0.
\end{align*}
\begin{lemma}
For all real numbers $p, a, t$ we have 
\begin{align}\label{utol9}
N(p+a, t+a^{2}) + N(p-a,t+a^{2})\geq 2N(p,t)
\end{align}
provided that $p \pm a >0$ and $t \geq 0$. 
\end{lemma}
\begin{proof}
First we verify that $N(p,t)$ satisfies backward heat equation 
\begin{align}\label{heateq0}
\frac{N_{pp}}{2}+N_{t}=0.
\end{align}
Indeed, we have 
\begin{align}
N_{pp}+2N_{t} &= -\frac{1}{p^{2}}+\frac{G''(p/\sqrt{t})}{t}-G'(p/\sqrt{t})pt^{-3/2}  \nonumber \\
&=\frac{1}{p^{2}} \left( -1 + s^{2}G''(s) - s^{3}G'(s)\right), \label{der0}
\end{align}
where $s=p/\sqrt{t}$. Direct calculations show 
\begin{align}
G'(s) = -e^{s^{2}/2}\int_{s}^{\infty}u^{-2}e^{-\frac{u^{2}}{2}}du, \quad G''(s) = s^{-2}-se^{s^{2}/2}\int_{s}^{\infty}u^{-2}e^{-\frac{u^{2}}{2}}du. \label{der1}
\end{align}
Substituting (\ref{der1}) into (\ref{der0}) we see that the expression in (\ref{der0}) is zero. 

Next, we claim that $t \mapsto N(p,t)$ is concave. Indeed, 
\begin{align*}
N_{t} &=-\frac{1}{2}pt^{-3/2} G'(p/\sqrt{t}),\\
N_{tt} &=\frac{1}{4}p^{2}t^{-3}G''(p/\sqrt{t}) +\frac{3}{4}pt^{-5/2}G'(p/\sqrt{t}) =\frac{1}{4t^{2}}\left[ s^{2} G''(s)+3sG'(s)\right].
\end{align*}
Since $N_{pp}+2N_{t}=0$ we have $G'' = s^{-2}+sG'(s)$ by (\ref{der0}), therefore, the sign of $N_{tt}$ coincides with the sign of $1+(s^{3}+3s)G'(s)$. Using (\ref{der1}) it suffices to show that 
\begin{align*}
\varphi(s):= \frac{e^{-s^{2}/2}}{s^{3}+3s}-\int_{s}^{\infty}s^{-2} e^{-s^{2}/2}\leq 0  \quad \text{for all} \quad s \geq 0. 
\end{align*}
We have $\varphi(\infty)=0$, and 
\begin{align*}
\varphi'(s) = e^{-s^{2}/2}\left[-\frac{1}{3+s^{2}}-\frac{3+3s^{2}}{(3s+s^{3})^{2}} +\frac{1}{s^{2}}\right] = \frac{6e^{s^{2}/2}}{(3s+s^{2})^{2}}\geq 0,
\end{align*}
thereby $\varphi(s) \leq 0$, and hence $t \mapsto N(p,t)$  is concave for $t\geq 0$.

Next, consider the process 
\begin{align*}
X_{s} = N(p+B_{s},t+s),
\end{align*}
where $B_{s}$ is the standard Brownian motion starting at zero. It follows from Ito's formula that $X_{s}$ is a martingale. Indeed,  we have 
\begin{align*}
dX_{s} = N_{s} ds+N_{p} dB_{s} +\frac{1}{2}V_{pp} ds  \stackrel{(\ref{heateq0})}{=}  N_{p} dB_{s}.
\end{align*}
Define the stopping time 
\begin{align*}
\tau = \inf\{ s \geq 0 : B_{s} \notin(-a,a)\}.
\end{align*}
Set $Y_{s} = Y_{\min\{s, \tau\}}$ for $s\geq 0$. Clearly $Y_{s}$ is a martingale.  On the one hand $Y_{0} = N(p,t)$. On the other hand 
\begin{align*}
&\mathbb{E} Y_{\infty} = \mathbb{E} N(p+B_{\tau}, t+\tau) = \\
&\mathbb{E}(N(p-a, t+\tau) | B_{\tau}=-a) \mathbb{P}(B_{\tau}=-a)+\mathbb{E}(N(p+a, t+\tau) | B_{\tau}=-a) \mathbb{P}(B_{\tau}=-a)\\
&\stackrel{\mathrm{concavity} \, \, t \mapsto N(p,t)}{\leq}\frac{1}{2}\left[  N(p-a, t+\mathbb{E} (\tau| B_{\tau}=-a))+N(p+a, t+\mathbb{E}( \tau| B_{\tau}=a))\right].
\end{align*}
Finally, as $B^{2}_{s}-s$ is a martingale, we have $0=\mathbb{E} (B^{2}_{\tau}-\tau) = a^{2}- \mathbb{E} \tau$. By symmetry we obtain $\mathbb{E}( \tau| B_{\tau}= - a) = \mathbb{E}( \tau| B_{\tau}= a) =a^{2}$. Thus the lemma  follows from the optional stopping theorem. 

\end{proof}

Before we complete the proof of Theorem~\ref{mt3} let us make a remark. If $N(p,t)$ is an arbitrary smooth function satisfying the backwards heat equation (\ref{heateq0}) and the inequality (\ref{utol9}) then $t \mapsto N(p,t)$ must be concave. In other words, the concavity of $t \mapsto N(p,t)$ is necessary and sufficient for the inequality (\ref{utol9}) to hold provided that $N$ solves backwards heat equation. Indeed, let $r(a)=N(p+a,t+a^{2})$. We have 
\begin{align*}
r'(a) &=N_{p}+2aN_{t},\\
r''(a) &= N_{pp}+4aN_{pt}+2N_{t}+4a^{2}N_{tt}\stackrel{(\ref{heateq0})}{=}4aN_{pt}+4a^{2}N_{tt},\\
r'''(a) &= 4N_{pt}+4aN_{ppt}+8a^{2}N_{ptt}+8aN_{tt}+4a^{2}N_{ttp}+8a^{3}N_{ttt}\\
&\stackrel{(\ref{heateq0})}{=}4N_{pt}+12a^{2}N_{ptt}+8a^{3}N_{ttt},\\
r''''(a) &=4N_{ptp}+8aN_{ptt}+24a N_{ptt}+12a^{2}N_{pttp}+24a^{3}N_{pttt}\\
 &+24a^{2}N_{ttt}+8a^{3}N_{tttp}+16a^{4}N_{tttt}\\
 &\stackrel{(\ref{heateq0})}{=}4N_{ppt}+32aN_{ptt}+32a^{3}N_{pttt}+16a^{4}N_{tttt}.
\end{align*}
By Taylor's formula we have 
\begin{align*}
&N(p+a, t+a^{2})+N(p-a, t+a^{2})=r(a)+r(-a) =2r(0)+r''(0) a^{2}+r''''(0) \frac{a^{4}}{12}+o(a^{4})\\
&=2N(p,t)+N_{ppt}(p,t)\frac{a^{4}}{3}+o(a^{4}) \stackrel{(\ref{heateq0})}{=} 2N(p,t) - N_{tt}(p,t) \frac{2a^{4}}{3}+o(a^{4}).
\end{align*}
Thus it follows from (\ref{heateq0}) that the limit 
\begin{align*}
\lim_{a \to 0} \frac{N(p+a, t+a^{2})+N(p-a,t+a^{2})-2N(p,t)}{a^{4}} = -  \frac{2}{3} N_{tt}
\end{align*}
is non-negative, i.e., $t \mapsto N(p,t)$ is concave. 

\vskip1cm 

Now we are ready to complete the proof of Theorem~\ref{mt3}. Let $N\geq 0$ be such that $\xi_{N}=\xi_{N+1}=...=\xi$. We have  
\begin{align*}
&\mathbb{E} N(\xi, [\xi]) =\mathbb{E} N(\xi_{N}, [\xi_{N}]) =\\
&\mathbb{E} \left(\mathbb{E}\left(N(\xi_{0}+(\xi_{1}-\xi_{0})+\ldots+(\xi_{N}-\xi_{N-1}), (\xi_{1}-\xi_{0})^{2}+\ldots+(\xi_{N}-\xi_{N-1})^{2}) |\mathcal{F}_{N-1}\right) \right).
\end{align*}
Notice that random variables $\eta = \xi_{0}+(\xi_{1}-\xi_{0})+\ldots+(\xi_{N-1}-\xi_{N-2})$ and $\zeta = (\xi_{1}-\xi_{0})^{2}+\ldots+(\xi_{N-1}-\xi_{N-2})^{2}$ are $\mathcal{F}_{N-1}$ measurable. Yet on each atom $Q$  of $\mathcal{F}_{N-1}$ the random variable $\xi_{N-1}-\xi_{N}$ takes values $\pm A$ with equal probabilities $|Q|/2$.   Then it follows from (\ref{utol9}) that 
\begin{align*}
\mathbb{E} N(\xi_{N}, [\xi_{N}])  \geq N(\xi_{N-1}, [\xi_{N-1}]).
\end{align*}
Iterating this inequality  and using the boundary value $N(p,0) =\log p$ for $p>0$ we obtain $\mathbb{E} N(\xi, [\xi]) \geq \mathbb{E}N(\xi_{0}, 0) = \ln \mathbb{E} \xi$. This finishes the proof of Theorem~\ref{mt3}. 

\vskip1cm
The inequality (\ref{sled2}) follows from 
\begin{lemma}\label{mokla1}
We have 
\begin{align}\label{twoss}
\log(1+y^{-2}) \geq \int_{y}^{\infty}\int_{x}^{\infty} e^{\frac{-t^{2}+x^{2}}{2}}t^{-2}dtdx \geq \frac{1}{3} \log (1+y^{-2})\end{align}
holds for all $y>0$. 
\end{lemma}
\begin{proof}
For a positive constant $C>0$ consider a map 
\begin{align*}
h(y;C) = \int_{y}^{\infty}\int_{x}^{\infty} e^{\frac{-t^{2}+x^{2}}{2}}t^{-2}dtdx - C\log(1+y^{-2}), \quad y>0.
\end{align*}
Notice that $h (\infty;C)=0$. To prove the right hand side inequality in (\ref{twoss}) (or the left hand side in (\ref{twoss})) it suffices to show 
\begin{align}\label{nabiji}
h_{y}(y;C) = -\int_{y}^{\infty} \frac{e^{\frac{-t^{2}+y^{2}}{2}}}{t^{2}}dt +2C \frac{1}{y^{3}+y} \leq 0
\end{align}
for $C=1/3$ (or $h_{y}(y;C) \geq 0$ for $C=1$). Next, consider $\psi(y;C) = e^{-y^{2}/2}h'(y;C) = -\int_{y}^{\infty}e^{-t^{2}/2}t^{-2}dt+2C \frac{e^{-y^{2}/2}}{y^{3}+y}$. Clearly $\psi(\infty; C)=0$. To show (\ref{nabiji}) for $C=\frac{1}{3}$ (or its reverse inequality when  $C=1$) it suffices to verify that 
\begin{align*}
\psi_{y} (y;C) &= \frac{Ce^{-y^{2}/2}}{y^{2}}\left(\frac{1}{C}-2\frac{(3y^{2}+1)}{(y^{2}+1)^{2}} - \frac{y^{2}}{y^{2}+1} \right)\\
 &=\frac{Ce^{-y^{2}/2}}{y^{2}}\left(\frac{1}{C}-1 +\frac{4}{(y^{2}+1)^{2}} - \frac{5}{y^{2}+1} \right)  \geq 0
\end{align*}
for $C=\frac{1}{3}$ (or the reverse inequality for $C=1)$. Let $s=(y^{2}+1)^{-1} \in [0,1]$. Then $-1+4t^{2}-5t$ is minimized on $[0,1]$ when $t=5/8$ and its minimal value is $-41/16$ (or maximized on $[0,1]$ when $t=0$ and its maximal value is $-1$).
 The lemma is proved.

\end{proof}

\section{Concluding remarks}

One may ask how we guessed $N(p,t)$ which played an essential role in the proof of Theorem~\ref{mt3}. There is a general argument \cite{INV2} which informally says that ``any estimate in Gauss space  (or more generally on the hamming cube) involving $f$ and its gradient, has a corresponding {\em dual estimate} for a stopped Brownian motion and its quadratic variation (or more generally dyadic square function)''. For example, to prove the inequality (\ref{oo1}), there was a certain function $M(x,y)$ used in the proof. This function satisfies Monge--Ampere type PDE 
\begin{align}\label{mga}
\mathrm{det} \begin{pmatrix}
M_{xx}+\frac{M_{y}}{y} & M_{xy} \\
M_{xy} & M_{yy}
\end{pmatrix} = 0
\end{align}
with a boundary condition $M(x,0)=\log(x)$ so that the matrix in (\ref{mga}) is positive definite. Suppose we would like to solve the PDE (\ref{mga}) in general. Using exterior differential systems (see details in~\cite{IV3}) the PDE may be linearized to the backwards heat equation, namely, locally the solutions can be parametrized as 
\begin{align*}
\begin{cases}
&M(x,y)=-px+\sqrt{t} y + u(p,t),\\
&x=-u_{p}(p,t),\\
&y=2\sqrt{t} u_{t}(p,t),
\end{cases}
\end{align*}
where $u$ satisfies backwards heat equation 
\begin{align*}
\begin{cases}
&u_{t}+\frac{u_{pp}}{2}=0, \\
&u(M_{x}(x,0),0)=M(x,0)-xM_{x}(x,0),
\end{cases}
\end{align*}
with $t\geq 0$, and $p \in \Omega \subset \mathbb{R}$.  An important observation is that if $u$ happens to satisfy 
\begin{align*}
u(p+a, t+a^{2})+u(p-a,t+a^{2}) \geq 2u(p,t)
\end{align*}
then under some additional assumptions on $u$, one expects an identity 
\begin{align*}
M(x,y) = \sup_{t} \inf_{p } \left\{ -px+\sqrt{t} y + u(p,t) \right\} =\inf_{p}  \sup_{t} \left\{ -px+\sqrt{t} y + u(p,t) \right\},
\end{align*}
which, if true,  implies that $M$ satisfies 4-point inequality (\ref{two-pp1}), see \cite{INV2,INV3} for more details. These functions $M(x,y)$ and $u(p,t)$ we call dual to each other. One may verify that for our particular $M$ defined by (\ref{mfun}), the corresponding dual $u(p,t)$ is 
\begin{align*}
u(p,t) =1+\log(-p)+\int_{-p/\sqrt{t}}^{\infty}\int_{s}^{\infty}r^{-2}e^{\frac{-r^{2}+s^{2}}{2}}drds, \quad p<0, \quad t\geq 0, 
\end{align*}
which coincides with $N(p,t)$ after  subtracting $1$,  and reflecting the variable $p$. 

Using this approach, one could try to prove 4-point inequality (\ref{two-pp1}) which would imply 
$$
\mathbb{E} M(g, |Dg|) \geq M(\mathbb{E}g, 0) \quad \text{for all} \quad g :\{-1,1\}^{n} \to \mathbb{R}_{+}.
$$
So, one may hope to obtain (\ref{mthii}) on the hamming cube after substituting $g=e^{f}$. 
However, we did not proceed with this path on the unfortunate grounds that the chain rule misbehaves on the hamming cube, i.e., the identity $\frac{|De^{f}|}{e^{f}} = |Df|$ does not hold. Therefore, to prove  (\ref{oo1})  on the hamming cube  perhaps different ideas are needed. 

Our last remark is that one may provide  another proof of (\ref{sled2}) using a simpler function compared to $N$ (what we call   {\em the  supersolution}). Indeed,  consider 
\begin{align*}
N^{\mathrm{sup}}(p,t) = \frac{1}{2}\log(p^{2}+t), \quad t \geq 0, p>0.
\end{align*}
Notice that 
\begin{align}
&N^{\mathrm{sup}}(p,0)=\log(p), \label{mk1}\\
&\frac{N_{pp}}{2}+N_{t} = \frac{t+t^{2}}{2(p^{2}+t)^{2}} \geq 0, \label{mokla2}\\
&N^{\mathrm{sup}}_{tt} = -\frac{1}{2}\frac{1}{(t+p^{2})^{2}}\leq 0.\label{mk2}
\end{align}
Using the same argument as in the proof of (\ref{martin}) we verify that $N^{\mathrm{sup}}(p,t)$ satisfies (\ref{utol9}). Notice that $N^{\mathrm{sup}}$ does not solve the backwards heat equation, however, due to inequality (\ref{mokla2})  the stochastic  process $X_{s}$ constructed in the proof of (\ref{utol9}) will be submartingale which is sufficient for the proof of  (\ref{utol9}). Thus we obtain 
\begin{align*}
\log \mathbb{E} \xi_{\infty} - \mathbb{E} \log \xi_{\infty} \leq \frac{1}{2} \mathbb{E} \log\left( 1+\frac{[\xi_{\infty}]}{\xi_{\infty}^{2}}\right)
\end{align*}
which improves on (\ref{sled2}) by a factor of $1/2$. 

The supersolution $N(p,t)$ was guessed from the form of the inequality (\ref{sled2})  by considering $\log(p)+C \log(1+\frac{t}{p^{2}})$ and choosing an optimal constant $C$ (in our case $C=1/2$ worked well). It was a good coincidence that such $N^{\mathrm{sup}}$ satisfies (\ref{mk1})
, (\ref{mokla2}), and (\ref{mk2}).  However, if one tries to construct a supersolution to the inequality (\ref{oo1}) one may hope that, by chance, 
\begin{align*}
M(x,y) = \log(x)+ C e^{\frac{y^{2}}{2x^{2}}}(1+y/x)^{-1}
\end{align*}
may work for some positive $C$. A direct calculation shows that there is no positive constant $C$ such that  the inequality (\ref{kuku}) holds true.

\subsection*{Acknowledgments} The authors thank Sergey Bobkov and Changfeng Gui for helpful discussions and comments. P.I. was supported in part by NSF grants DMS-1856486, DMS-2052645 and CAREER-DMS-1945102 and DMS-2052865.

\end{document}